\documentclass{amsart}
\usepackage{moreverb,url}
\usepackage{amssymb,amsmath,amsfonts,mathrsfs}
\usepackage{bm}
\usepackage{float}
\usepackage{verbatim}
\usepackage[square,numbers]{natbib}
\usepackage{stmaryrd}
\usepackage{fancyhdr}
\usepackage{syntax,etoolbox}
\usepackage{graphicx}
\usepackage[dvipsnames]{xcolor}
\usepackage{hyperref}
\usepackage{subcaption}
\usepackage[font=scriptsize, labelfont=bf]{caption}
\hypersetup{colorlinks=true,allcolors=blue}
\usepackage[square,numbers]{natbib}
\usepackage[left=2.5cm, right=2.5cm, bottom=3cm]{geometry}
\usepackage{lineno}

\newtheorem{theorem}{Theorem}[section]

\newtheorem{innercustomgeneric}{\customgenericname}
\providecommand{\customgenericname}{}
\newcommand{\newcustomproblem}[2]{%
	\newenvironment{#1}[1]
	{%
		\renewcommand\customgenericname{#2}%
		\renewcommand\theinnercustomgeneric{##1}%
		\innercustomgeneric
	}
	{\endinnercustomgeneric}
}

\newcustomproblem{customprob}{Problem}

\newcommand*{\bqed}{\hfill\ensuremath{\blacksquare}}%

\begin{document}
	
	
	\title[Density of smooth inactive displacements]{On the feasible set of a membrane shell under confinement. Smooth inactive admissible displacements are dense}
	
	
	\author{Paolo Piersanti}
	\address{School of Science and Engineering, The Chinese University of Hong Kong (Shenzhen), 2001 Longxiang Blvd., Longgang District, Shenzhen, China}
	\email{ppiersanti@cuhk.edu.cn}

\begin{abstract}
	We establish a density property for the obstacle problem of a linearly elastic elliptic membrane shell required to remain confined in a prescribed half-space. For a bounded Lipschitz reference domain, and under the sole assumption that the reference configuration lies at a strictly positive distance from the confining plane, we prove that smooth admissible displacements for which the confinement constraint is inactive (satisfied with a strictly positive margin) are dense in the set of all admissible displacements. This removes the additional condition on the unit normal to the mid-surface required in previous work, and the approximation scheme is elementary and adaptable to other unilateral obstacle problems.
\end{abstract}

\maketitle

\section{Geometrical preliminaries} \label{Sec:1}
For details about the classical notions of differential geometry used throughout the paper see, e.g.~\cite{Ciarlet2000}.

Greek indices and exponents vary in the set $\{1,2\}$, while Latin indices or exponents, except when they are used for indexing sequences, vary in the set $\{1,2,3\}$. The summation convention with respect to repeated indices or exponents is systematically used in conjunction with these two rules. We denote by $\mathbb{E}^3$ the three-dimensional Euclidean space. The notation $\delta^j_i$ designates the Kronecker symbol.

Given an open subset $\Omega$ of $\mathbb{R}^d$, the notation $\mathcal{D}(\Omega)$ denotes the space of all functions that are infinitely differentiable over $\Omega$ and have compact supports in $\Omega$. Given a vector space $X$, the corresponding product space $X \times X \times X$ is denoted by $\bm{X}$. The abbreviations \emph{``a.a.''} and \emph{``a.e.''} mean \emph{almost all} and \emph{almost everywhere}, respectively.

Let $\omega$ be a Lipschitz domain in $\mathbb{R}^2$ (cf., e.g. Section~8.2 in~\cite{Ciarlet2025}), let $y = (y_\alpha)$ denote a generic point in $\overline{\omega}$. A mapping $\bm{\theta} \in \mathcal{C}^1(\overline{\omega}; \mathbb{E}^3)$ is an \emph{immersion} if the two vectors
$$
\bm{a}_\alpha(y) := \partial_\alpha \bm{\theta}(y)
$$
are linearly independent at each point $y \in \overline{\omega}$. Then the image $\bm{\theta}(\overline{\omega})$ of the set $\overline{\omega}$ under the mapping $\bm{\theta}$ is a \emph{surface in} $\mathbb{E}^3$, equipped with $y_1, y_2$ as its \emph{curvilinear coordinates}. Given any point $y\in \overline{\omega}$, the vectors $\bm{a}_\alpha(y)$ span the \emph{tangent plane} to the surface $\bm{\theta}(\overline{\omega})$ at the point $\bm{\theta}(y)$, the unit vector
$$
\bm{a}_3(y) := \frac{\bm{a}_1(y) \wedge \bm{a}_2(y)}{|\bm{a}_1(y) \wedge \bm{a}_2(y)|}
$$
is normal to $\bm{\theta}(\overline{\omega})$ at $\bm{\theta} (y)$, the three vectors $\bm{a}_i(y)$ form the \emph{covariant} basis at $\bm{\theta}(y)$, and the three vectors $\bm{a}^j(y)$ defined by the relations
$$
\bm{a}^j(y) \cdot \bm{a}_i(y) = \delta^j_i
$$
form the \emph{contravariant} basis at $\bm{\theta}(y)$; note that the vectors $\bm{a}^\beta (y)$ also span the tangent plane to $\bm{\theta}(\overline{\omega})$ at $\bm{\theta}(y)$ and that $\bm{a}^3(y) = \bm{a}_3(y)$. The special notation $\omega_1 \subset \subset \omega$ means that $\overline{\omega_1} \subset \omega$ and $\textup{dist}(\partial\omega,\partial\omega_1):=\min\{|x-y|;x \in \partial\omega \textup{ and } y \in \partial\omega_1\}>0$.

In this paper we consider a specific \emph{obstacle problem} for linearly elastic elliptic membrane shells (cf., e.g., Chapter~4 in~\cite{Ciarlet2000}), where the shell is subjected to a \emph{confinement condition}, expressing that any \emph{admissible displacement vector field} $\eta_i \bm{a}^i$ must be such that all the points of the corresponding deformed configuration remain in a prescribed \emph{half-space} of the form
$$
\mathbb{H}:=\{x\in\mathbb{E}^3;\boldsymbol{Ox}\cdot\bm{q}\ge 0\},
$$
where $\bm{q}$ is a \emph{unit-vector} which is given once and for all, and which is thus orthogonal to the plane associated with the half-space where the linearly elastic shell is required to remain confined. In other words, any admissible displacement field must satisfy
\begin{equation}
	\label{cc-original}
	\left(\bm{\theta}(y)+\eta_i(y)\bm{a}^i(y)\right)\cdot\bm{q}\ge 0,\quad\textup{ for a.a. }y\in\omega.
\end{equation}

We denote by $\bm{U}_M(\omega)$ the set:
\begin{equation*}
	\bm{U}_M(\omega):=\{\bm{\eta}=(\eta_i)\in H^1_0(\omega)\times H^1_0(\omega)\times L^2(\omega);(\bm{\theta}(y)+\eta_i(y)\bm{a}^i(y))\cdot\bm{q}\ge 0 \textup{ for a.a. }y\in\omega\}.
\end{equation*}

The variational problem governing the deformation of a linearly elastic elliptic membrane shell subjected to the confinement condition~\eqref{cc-original} takes the following form (viz.~\cite{CiaMarPie2018b,CiaMarPie2018, Pie-2021}).

\begin{customprob}{$\mathcal{P}_M(\omega)$}
	\label{problemLim}
	Find $\bm{\zeta} \in \bm{U}_M(\omega)$ that satisfies the following variational inequalities
	$$
	\int_{\omega} a^{\alpha \beta \sigma \tau} \gamma_{\sigma \tau}(\bm{\zeta}) \gamma_{\alpha \beta} (\bm{\eta} - \bm{\zeta}) \sqrt{a} \mathrm{d} y \ge \int_{\omega} p^i (\eta_i - \zeta_i) \sqrt{a} \mathrm{d} y,
	$$
	for all $\bm{\eta} = (\eta_i) \in \bm{U}_M(\omega)$.
	\bqed
\end{customprob}

In~\cite{CiaMarPie2018}, in order to justify the rigorous asymptotic analysis departing from the three-dimensional counterpart of the obstacle problem under consideration, based on the classical three-dimensional energy for linearly elastic bodies, we needed to resort to the ``density property'', asserting that the set $\bm{U}_M(\omega)\cap\bm{H}^1_0(\omega)$ is dense in $\bm{U}_M(\omega)$ with respect to the norm $\|\cdot\|_{H^1(\omega)\times H^1(\omega)\times L^2(\omega)}$. In~\cite{CiaMarPie2018}, it was shown that if $\min_{y\in\overline{\omega}}(\bm{\theta}(y)\cdot\bm{q})>0$ and $\min_{y\in\overline{\omega}}(\bm{a}^3(y)\cdot\bm{q})>0$ then the ``density property'' holds. The ``density property'' was exploited in~\cite{MeiPie2024,Pie-2021-2} for establishing the higher interior regularity of the solution of Problem~\ref{problemLim}.
In this paper, we establish the validity of the ``density property'' \emph{without} the assumption that $\min_{y\in\overline{\omega}}(\bm{a}^3(y)\cdot\bm{q})>0$.

\begin{theorem}
	Let $\omega$ be a bounded Lipschitz domain. Assume that the immersion $\bm{\theta}\in\mathcal{C}^3(\overline{\omega};\mathbb{E}^3)$ satisfies:
	\begin{equation*}
		\min_{y\in\overline{\omega}}(\bm{\theta}(y)\cdot\bm{q})>0.
	\end{equation*}
	
	Then, it results that $(\bm{U}_M(\omega)\cap\bm{\mathcal{D}}(\omega))$ is dense in $\bm{U}_M(\omega)$ with respect to the norm $\|\cdot\|_{H^1(\omega)\times H^1(\omega)\times L^2(\omega)}$.
\end{theorem}
\begin{proof}
	We divide the proof into six parts, numbered $(i)$--$(vi)$.
	
	$(i)$ \emph{The set
		\begin{equation*}
			\begin{aligned}
				\bm{U}_M^{(1)}(\omega):=\{&\bm{\eta}\in H^1_0(\omega)\times H^1_0(\omega)\times L^2(\omega);\\
				&\quad\textup{ there exists }\delta=\delta(\bm{\theta},\bm{\eta})>0 \textup{ such that }(\bm{\theta}+\eta_i\bm{a}^i)\cdot\bm{q}\ge\delta \textup{ a.e. in }\omega\},
			\end{aligned}
		\end{equation*}
		is dense in $\bm{U}_M(\omega)$ with respect to the norm $\|\cdot\|_{H^1(\omega)\times H^1(\omega)\times L^2(\omega)}$}. Let $\bm{\eta}\in\bm{U}_M(\omega)$ be given.
	For each integer $k\ge 1$, define the sequence $\{\bm{\eta}_k\}_{k=1}^\infty$ as follows:
	\begin{equation*}
		\bm{\eta}_k:=\left(1-\dfrac{1}{k}\right)\bm{\eta}.
	\end{equation*}
	
	It is clear that $\bm{\eta}_k=(\eta_{i,k})\in H^1_0(\omega)\times H^1_0(\omega)\times L^2(\omega)$ for each integer $k\ge 1$ and that $\|\bm{\eta}_k-\bm{\eta}\|_{H^1(\omega)\times H^1(\omega)\times L^2(\omega)}=k^{-1}\|\bm\eta\|_{H^1(\omega)\times H^1(\omega)\times L^2(\omega)}\to0$ as $k\to\infty$. Additionally, a direct calculation gives:
	\begin{equation*}
		(\bm{\theta}(y)+\eta_{i,k}(y)\bm{a}^i(y))\cdot\bm{q}\ge\dfrac{1}{k}(\min_{y\in\overline{\omega}}(\bm{\theta}(y)\cdot\bm{q}))>0, \quad\textup{ for a.a. }y\in\omega.
	\end{equation*}
	
	The proof is complete by letting $\delta:=\frac{1}{k}(\min_{y\in\overline{\omega}}(\bm{\theta}(y)\cdot\bm{q}))$.
	
	$(ii)$ \emph{Let $f:\mathbb{R}\to\mathbb{R}$ be a Lipschitz\footnote{The Lipschitz constant is not needed for this item. It is exploited in the next one. Here the Lipschitz continuity of $f$ ensures that $S_\tau$ is a Lipschitz domain so that the trace of $v$ on the graph of $f$ is well-defined.} continuous function. For every $\tau>0$, let $S_{\tau}:=\{(x',s)\in\mathbb{R}^2:f(x')<s<f(x')+\tau\}$.
		Then, for every $v\in H^1(S_{\tau})$ whose trace on the graph $\Gamma:=\{(x',f(x')):x'\in\mathbb{R}\}$ vanishes, one has:
		\begin{equation}
			\label{eq:strip}
			\int_{S_{\tau}}|v|^2\mathrm{d}x'\mathrm{d}s\le\tau^2\int_{S_{\tau}}|\nabla v|^2\mathrm{d}x'\mathrm{d}s.
	\end{equation}}

	Fix $\tau>0$. Let $\Phi:\mathbb{R}^2\to\mathbb{R}^2$ be the shear $\Phi(x',s):=(x',s-f(x'))$, which is bi-Lipschitz since $f$ is Lipschitz, has Jacobian determinant equal to $1$, and maps $S_{\tau}$ onto the straight strip $\mathbb{R}\times(0,\tau)$. Set $w:=v\circ\Phi^{-1}\in H^1(\mathbb{R}\times(0,\tau))$. Then $w$ has trace zero on $\mathbb{R}\times\{0\}$, and by Fubini's theorem (cf., e.g., Theorem~1.15-6$(b)$ in~\cite{Ciarlet2025}) together with the characterisation of weak differentiability via absolute continuity along lines (cf., e.g., Theorem~11.45 in~\cite{Leoni2017}), for a.e. $x'\in\mathbb{R}$ the map $t\mapsto w(x',t)$ belongs to $H^1((0,\tau))$ with $w(x',0)=0$. Hence,
	\begin{equation*}
	w(x',t)=\int_0^t\partial_t w(x',r)\mathrm{d}r,
	\end{equation*}
	and H\"{o}lder's inequality gives:
	\begin{equation*}
	|w(x',t)|^2\le t\int_0^t|\partial_t w(x',r)|^2\mathrm{d}r\le\tau\int_0^\tau|\partial_t w(x',r)|^2\mathrm{d}r.
	\end{equation*}
	
	Integrating first in $t\in(0,\tau)$ and then in $x'\in\mathbb{R}$ yields:
	\begin{equation*}
	\int_{\mathbb{R}\times(0,\tau)}|w|^2\mathrm{d}x'\mathrm{d}t\le\tau^2\int_{\mathbb{R}\times(0,\tau)}|\partial_t w|^2\mathrm{d}x'\mathrm{d}t.
	\end{equation*}
	
	It remains to transfer this estimate back to $v$ on the curved strip $S_{\tau}$. To this end, note that $\Phi$ is such that $|\det J_\Phi|=1$. The change-of-variables formula gives
	\begin{equation}
		\label{eq:measure}
		\int_{\mathbb{R}\times(0,\tau)}|w|^2\mathrm{d}x'\mathrm{d}t
		=\int_{S_{\tau}}|w\circ\Phi|^2\mathrm{d}x'\mathrm{d}s
		=\int_{S_{\tau}}|v|^2\mathrm{d}x'\mathrm{d}s,
	\end{equation}
	since $w\circ\Phi=v$. Likewise, differentiating the identity $w(x',t)=v(x',t+f(x'))$ in the variable $t$ gives
	\begin{equation*}
	\partial_t w(x',t)=\partial_s v(x',t+f(x')),
	\end{equation*}
	that is, $\partial_t w=(\partial_s v)\circ\Phi^{-1}$. Applying the same change of variables~\eqref{eq:measure} to $|\partial_t w|^2$ gives:
	\begin{equation}
		\label{eq:measure2}
		\int_{\mathbb{R}\times(0,\tau)}|\partial_t w|^2\mathrm{d}x'\mathrm{d}t=\int_{S_{\tau}}|\partial_s v|^2\mathrm{d}x'\mathrm{d}s.
	\end{equation}
	
	Finally, since $\partial_s v$ is one of the two components of the gradient of $v$,
	\begin{equation*}
	|\partial_s v|\le|\nabla v|, \quad\textup{ a.e. in }S_{\tau},
	\end{equation*}
	so that
	\begin{equation}
		\label{eq:grad}
		\int_{S_{\tau}}|\partial_s v|^2\mathrm{d}x'\mathrm{d}s\le\int_{S_{\tau}}|\nabla v|^2\mathrm{d}x'\mathrm{d}s.
	\end{equation}
	
	Combining the inequality
	\begin{equation*}
	\int_{\mathbb{R}\times(0,\tau)}|w|^2\mathrm{d}x'\mathrm{d}t\le\tau^2\int_{\mathbb{R}\times(0,\tau)}|\partial_t w|^2\mathrm{d}x'\mathrm{d}t
	\end{equation*}
	with~\eqref{eq:measure}, \eqref{eq:measure2} and \eqref{eq:grad} yields~\eqref{eq:strip}, and the proof is complete.
	
	$(iii)$ \emph{Let $f:\mathbb{R}\to\mathbb{R}$ be a Lipschitz continuous function with Lipschitz constant $L>0$, and let $\Gamma:=\{(x',f(x')):x'\in\mathbb{R}\}$. Then, for every $p=(x',s)$ satisfying $g(p):=s-f(x')\ge0$, one has:
		\begin{equation}\label{eq:comp}
			\dfrac{g(p)}{2\max\{L,1\}}\le\textup{dist}(p,\Gamma)\le g(p).
	\end{equation}}
	
	For the upper bound, observe that the point $(x',f(x'))$ lies on $\Gamma$ and $|p-(x',f(x'))|=|s-f(x')|=g(p)$, for every $p=(x',s)$ satisfying $g(p)\ge 0$.
	
	For the lower bound, if $L=0$ the claim is immediate. Assume $L>0$ and fix $q=(x'',f(x''))\in\Gamma$. Then, it results:
	\begin{equation*}
	s-f(x'')=s-f(x')+f(x')-f(x'')\ge g(p)-L|x' -x''|.
	\end{equation*}
	
	Observe, on the one hand, that if $|x' -x''|\ge\frac{g(p)}{2L}$ then:
	\begin{equation*}
		|p-q|\ge|x' -x''|\ge\frac{g(p)}{2L}\ge\frac{g(p)}{2\max\{L,1\}}.
	\end{equation*}
	
	On the other hand, if $|x' -x''|<\frac{g(p)}{2L}$ then $s-f(x'')\ge g(p)-L|x' -x''|>\frac{g(p)}2>0$, so that it results:
	\begin{equation*}
		|p-q|\ge s-f(x'')>\frac{g(p)}2\ge\frac{g(p)}{2\max\{L,1\}}.
	\end{equation*}
	
	Since, in both cases, it results $|p-q|\ge\frac{g(p)}{2\max\{L,1\}}$, taking the infimum over $q\in\Gamma$ gives the lower bound in \eqref{eq:comp} and the proof is complete.
	
	$(iv)$ \emph{There exist constants $\delta_0>0$, $K\ge1$ and $C>0$ such that, for every $0<\delta\le\delta_0$ and every $u\in H^1_0(\omega)$, it results
		\begin{equation}\label{eq:layer}
			\int_{\omega^{\delta}}|u|^2\mathrm{d}y\le C\delta^2\int_{\omega^{K\delta}}|\nabla u|^2\mathrm{d}y,
		\end{equation}
		where $\omega^{\delta}:=\{y\in\omega;\textup{dist}(y,\partial\omega)<\delta\}$.}
	Since $\partial\omega$ is Lipschitz, there exist finitely many open sets $U_1,\dots,U_N$ covering $\partial\omega$. For each integer $1\le j\le N$, there exist an isometry $R_j:\mathbb{R}^2\to\mathbb{R}^2$, a bounded interval $I_j\subset\mathbb{R}$, numbers $a_j>0$ and $\varepsilon_j>0$, and a Lipschitz continuous function $f_j:\mathbb{R}\to\mathbb{R}$ (obtained by extending the local boundary chart to all of $\mathbb{R}$) with Lipschitz constant $L_j$ such that, in the coordinates $(x',s):=R_j^{-1}(y)$, it results:
	\begin{equation*}
	\begin{aligned}
		R_j^{-1}(U_j)&=\{(x',s);\ x'\in I_j\textup{ and } f_j(x')-\varepsilon_j<s<f_j(x')+a_j\},\\
		\omega\cap U_j&=R_j\left(\{(x',s);x'\in I_j\textup{ and }f_j(x')<s<f_j(x')+a_j\}\right),\\
		\partial\omega\cap U_j&=R_j\left(\{(x',f_j(x'));x'\in I_j\}\right).
	\end{aligned}
	\end{equation*}
	
	For each integer $1\le j\le N$, and for $0<\tau\le a_j$, define the local strip
	\begin{equation*}
	\Sigma_j(\tau):=R_j\left(\{(x',s);x'\in I_j,f_j(x')<s<f_j(x')+\tau\}\right),
	\end{equation*}
	and observe that, by the upper bound in~\eqref{eq:comp} and the fact that $R_j$ is an isometry, it results that $\Sigma_j(\tau)\subset\omega^{\tau}$ for every $0<\tau\le a_j$.
	
	Since $\partial\omega$ is compact and the sets $U_j$ are open, there exists $\lambda>0$ such that, for every $z\in\partial\omega$, the ball $B(z,\lambda)$ is contained in $U_j$ for some integer $1\le j\le N$. Define the number
	\begin{equation*}
		K:=2\left(\max_{1\le j\le N}\max\{L_j,1\}\right)\ge1,
	\end{equation*}
	and choose $\delta_0>0$ such that $\delta_0<\lambda$ and $K\delta_0\le\min_{1\le j\le N}a_j$.
	
	Fix $0<\delta\le\delta_0$ and fix $y\in\omega^{\delta}$. Choose $z\in\partial\omega$ such that $|y-z|=\textup{dist}(y,\partial\omega)<\delta$. Since $\delta<\lambda$, there exists an integer $1\le j\le N$ such that $y,z\in U_j$. Writing $z=R_j(x'_0,f_j(x'_0))$ and $y=R_j(x',s)$ with $x'\in I_j$ and $f_j(x')<s<f_j(x')+a_j$, an application of the lower bound in~\eqref{eq:comp} to the graph $\Gamma_j:=\{(x',f_j(x'));x'\in\mathbb{R}\}$ gives
	\begin{equation*}
	\frac{s-f_j(x')}{K}\le\frac{s-f_j(x')}{2\max\{L_j,1\}}\le\textup{dist}((x',s),\Gamma_j)\le\big|(x',s)-(x'_0,f_j(x'_0))\big|=|y-z|<\delta,
	\end{equation*}
	so that $s-f_j(x')<K\delta\le a_j$ or, equivalently, $y\in\Sigma_j(K\delta)$. Therefore, we have shown that:
	\begin{equation}
		\label{eq:cover}
		\omega^{\delta}\subset\bigcup_{j=1}^{N}\Sigma_j(K\delta).
	\end{equation}
	
	Let us now move on to show~\eqref{eq:layer}. Fix $u\in H^1_0(\omega)$. An application of~\eqref{eq:cover} gives:
	\begin{equation*}
	\int_{\omega^{\delta}}|u|^2\mathrm{d}y\le\sum_{j=1}^{N}\int_{\Sigma_j(K\delta)}|u|^2\mathrm{d}y.
	\end{equation*}
	
	Noticing that the proof is unchanged when $x'$ ranges over the interval $I_j$, we now apply~\eqref{eq:strip} in each local chart with $\tau=K\delta$ and $v:=u\circ R_j$.
	Observe that if the trace of $u$ vanishes on $\partial\omega$, then the trace of $v$ vanishes on the graph $\{(x',f_j(x'));x'\in I_j\}$ as well. In conclusion, we obtain that an application of~\eqref{eq:strip} in each local chart gives:
	\begin{equation}
		\label{chart:1}
		\int_{\Sigma_j(K\delta)}|u|^2\mathrm{d}y\le(K\delta)^2\int_{\Sigma_j(K\delta)}|\nabla u|^2\mathrm{d}y.
	\end{equation}
	
	Having chosen $K\delta\le a_j$, the upper bound in~\eqref{eq:comp} holds and we infer that $\Sigma_j(K\delta)\subset\omega^{K\delta}$ for all integers $1\le j\le N$.
	In light of this, summing~\eqref{chart:1} over $j=1,\dots,N$ gives
	\begin{equation*}
	\int_{\omega^{\delta}}|u|^2\mathrm{d}y\le(K\delta)^2\sum_{j=1}^{N}\int_{\Sigma_j(K\delta)}|\nabla u|^2\mathrm{d}y\le NK^2\delta^2\int_{\omega^{K\delta}}|\nabla u|^2\mathrm{d}y,
	\end{equation*}
	and the claim follows with $C:=NK^2$.
	
	$(v)$ \emph{The set
		\begin{equation*}
			\bm{U}_M^{(2)}(\omega):=\{\bm{\eta}\in \bm{U}_M^{(1)}(\omega);\bm{\eta}=\bm{0} \textup{ a.e. in a neighbourhood of }\partial\omega\},
		\end{equation*}
		is dense in $\bm{U}_M^{(1)}(\omega)$ with respect to the norm $\|\cdot\|_{H^1(\omega)\times H^1(\omega)\times L^2(\omega)}$}. Fix $\bm{\eta}\in\bm{U}_M^{(1)}(\omega)$ and let $\delta=\delta(\bm{\theta},\bm{\eta})>0$ be such that:
		\begin{equation*}
		(\bm{\theta}(y)+\eta_i(y)\bm{a}^i(y))\cdot\bm{q}\ge\delta,\quad\textup{ for a.a. }y\in\omega.
		\end{equation*}
		
		Let $d(y):=\textup{dist}(y,\partial\omega)$ and let $\psi:\mathbb{R}\to[0,1]$ be the function defined by:
		\begin{equation*}
		\psi(t):=
		\begin{cases}
		0,& t\le\frac12,\\
		2t-1,& \frac12<t<1,\\
		1,& t\ge1.
		\end{cases}
		\end{equation*}
		
		Observe that $\psi$ is Lipschitz continuous with $|\psi'|\le 2$ a.e. in $\mathbb{R}$. For each integer $k\ge1$ with $k^{-1}\le\delta_0$, define the function:
		\begin{equation*}
		\xi_k(y):=\psi(kd(y)),\qquad\textup{ for all }y\in\overline{\omega}.
		\end{equation*}
		
		Since $d$ is Lipschitz continuous with Lipschitz constant equal to $1$, and $\psi$ is Lipschitz continuous with Lipschitz constant equal to $2$, it results $\xi_k\in W^{1,\infty}(\omega)$. Moreover, for a.a. $y\in\omega$, it can be easily verified that $\xi_k$ satisfies the following properties
		\begin{equation}
			\label{xi:prop}
		\begin{aligned}
			0\le\xi_k\le 1&,\quad\textup{ in }\overline{\omega},\\
			\xi_k=1&,\quad\textup{ on }\{y\in\omega;d(y)\ge k^{-1}\},\\
			\xi_k=0&,\quad\textup{ on }\{y\in\omega;d(y)\le(2k)^{-1}\},\\
			|\nabla\xi_k|\le 2k&,\quad\textup{ a.e. in }\omega,\\
			\nabla\xi_k\neq 0&\quad\textup{ only on the open annulus }\{y\in\omega;(2k)^{-1}<d(y)<k^{-1}\}\subset\omega^{1/k},\\
			\textup{supp }\nabla\xi_k&=\{y\in\omega;(2k)^{-1}\le d(y)\le k^{-1}\}\subset\subset\omega,
		\end{aligned}
		\end{equation}
		from which it results that $\xi_k\in H^1_0(\omega)\cap W^{1,\infty}(\omega)$. Define the vector field:
		\begin{equation*}
		\bm{\eta}_k(y):=\xi_k(y)\bm{\eta}(y),\qquad\textup{ for a.a. }y\in\omega.
		\end{equation*}
		
		Since $\xi_k\in W^{1,\infty}(\omega)$ by~\eqref{xi:prop} and $\eta_{\beta}\in H^1_0(\omega)$, the characterisation of weak differentiability via absolute continuity along lines (cf., e.g., Theorem~11.45 in~\cite{Leoni2017}) gives $\eta_{\beta,k}:=\xi_k\eta_{\beta}\in H^1(\omega)$ and, moreover, the classical product formula holds. 
		
		Since, by~\eqref{xi:prop}, it results $\textup{supp }(\xi_k\eta_{\beta})\subset\subset\omega$, we obtain that $\eta_{\beta,k}\in H^1_0(\omega)$. Since $\eta_{3,k}:=\xi_k\eta_3\in L^2(\omega)$, it results $\bm{\eta}_k\in H^1_0(\omega)\times H^1_0(\omega)\times L^2(\omega)$. Let us now verify that $\bm{\eta}_k=(\eta_{i,k})$ verifies the constraint a.e. in $\omega$. For a.a. $y\in\omega$, it results:
		\begin{equation}
			\label{eq:membership}
			\begin{aligned}
				&(\bm{\theta}(y)+\eta_{i,k}(y)\bm{a}^i(y))\cdot\bm{q}=(1-\xi_k(y))\bm{\theta}(y)\cdot\bm{q}+\xi_k(y)(\bm{\theta}(y)+\eta_i(y)\bm{a}^i(y))\cdot\bm{q}\\
				&\ge(1-\xi_k(y))\left(\min_{y\in\overline{\omega}}(\bm{\theta}(y)\cdot\bm{q})\right)+\xi_k(y)\delta\ge\min\left\{\min_{y\in\overline{\omega}}(\bm{\theta}(y)\cdot\bm{q}),\delta\right\}>0.
			\end{aligned}
		\end{equation}
		
		Combining the fact that $\bm{\eta}_k=\bm{0}$ a.e. in $\{y\in\omega;d(y)<(2k)^{-1}\}$, \eqref{eq:membership} and $(\bm{\theta}(y)+\eta_{i,k}(y)\bm{a}^i(y))\cdot\bm{q}\ge\tilde{\delta}:=\min\left\{\min_{y\in\overline{\omega}}(\bm{\theta}(y)\cdot\bm{q}),\delta\right\}$ gives $\bm{\eta}_k\in\bm{U}_M^{(2)}(\omega)$.
		
		Let us now show that $\|\eta_{i,k}-\eta_i\|_{L^2(\omega)}\to 0$ as $k\to\infty$, for each $1\le i\le 3$. Since, by~\eqref{xi:prop}, the function $(1-\xi_k)$ has support in $\omega^{1/k}$ and $0\le1-\xi_k\le1$, since $|\eta_i|^2\in L^1(\omega)$ and since $|\omega^{1/k}|\to0$ as $k\to\infty$, the absolute continuity of the Lebesgue integral gives
		\begin{equation}
			\label{eq:L2}
			\int_{\omega}|\eta_{i,k}-\eta_i|^2\mathrm{d}y=\int_{\omega}(1-\xi_k)^2|\eta_i|^2\mathrm{d}y\le\int_{\omega^{1/k}}|\eta_i|^2\mathrm{d}y\to0,
		\end{equation}
		as $k\to\infty$, showing that $\|\eta_{i,k}-\eta_i\|_{L^2(\omega)}\to 0$ as $k\to\infty$, for each $1\le i\le 3$.
		
		Let us now show that $\|\partial_\alpha\eta_{\beta,k}-\partial_\alpha\eta_\beta\|_{L^2(\omega)}\to 0$ as $k\to\infty$, for each $\alpha,\beta\in\{1,2\}$.
		An application of the classical product formula to $\eta_{\beta,k}\in H^1_0(\omega)$ gives
		\begin{equation*}
		\partial_{\alpha}\eta_{\beta,k}(y)-\partial_{\alpha}\eta_{\beta}(y)=(\partial_{\alpha}\xi_k(y))\eta_{\beta}(y)+(\xi_k(y)-1)\partial_{\alpha}\eta_{\beta}(y),\qquad\textup{ for a.a. }y\in\omega,
		\end{equation*}
		whence
		\begin{equation*}
		\begin{aligned}
		\int_{\omega}|\partial_{\alpha}\eta_{\beta,k}-\partial_{\alpha}\eta_{\beta}|^2\mathrm{d}y
		&\le2\int_{\omega}|\nabla\xi_k|^2|\eta_{\beta}|^2\mathrm{d}y
		+2\int_{\omega}(1-\xi_k)^2|\partial_{\alpha}\eta_{\beta}|^2\mathrm{d}y .
		\end{aligned}
		\end{equation*}
		
		For the first term, observe that~\eqref{xi:prop} gives $|\nabla\xi_k|\le 2k$ a.e. in $\omega$, and that $\nabla\xi_k$ vanishes outside $\omega^{1/k}$. An application of item $(iv)$ combined with the absolute continuity of the Lebesgue integral gives
		\begin{equation*}
		\int_{\omega}|\nabla\xi_k|^2|\eta_{\beta}|^2\mathrm{d}y\le 4 k^2\int_{\omega^{1/k}}|\eta_{\beta}|^2\mathrm{d}y\le 4C\int_{\omega^{K/k}}|\nabla\eta_{\beta}|^2\mathrm{d}y\to0,\quad\textup{ as }k\to\infty.
		\end{equation*}
		
		For the second term, since $(1-\xi_k)$ has support in $\omega^{1/k}$, it results:
		\begin{equation*}
		\int_{\omega}(1-\xi_k)^2|\partial_{\alpha}\eta_{\beta}|^2\mathrm{d}y\le\int_{\omega^{1/k}}|\partial_{\alpha}\eta_{\beta}|^2\mathrm{d}y\to0,\quad\textup{ as }k\to\infty.
		\end{equation*}
		
		Hence $\partial_{\alpha}\eta_{\beta,k}\to\partial_{\alpha}\eta_{\beta}$ in $L^2(\omega)$ as $k\to\infty$ for $\alpha,\beta\in\{1,2\}$. Combining the latter convergence with the convergence~\eqref{eq:L2} we obtain $\eta_{\beta,k}\to\eta_{\beta}$ in $H^1(\omega)$ as $k\to\infty$ and the proof is complete.
		
		$(vi)$ \emph{The set $\bm{U}_M^{(3)}(\omega):=\bm{\mathcal{D}}(\omega)\cap\bm{U}_M^{(2)}(\omega)$ is dense in $\bm{U}_M^{(2)}(\omega)$  with respect to the norm $\|\cdot\|_{H^1(\omega)\times H^1(\omega)\times L^2(\omega)}$.} Let $\bm{\eta}\in\bm{U}_M^{(2)}(\omega)$ be given and let $\delta>0$ be such that $(\bm{\theta}(y)+\eta_i(y)\bm{a}^i(y))\cdot\bm{q}\ge\delta$, for a.a. $y\in\omega$.
		
		Since the support of $\bm{\eta}$ is compact in $\omega$, we can extend $\bm{\eta}$ by $\bm{0}$ outside of $\omega$ and denote this extension by $\tilde{\bm{\eta}}$. Since $\bm{\theta}\in\mathcal{C}^3(\overline{\omega};\mathbb{E}^3)$, the Whitney extension theorem provides $\tilde{\bm{\theta}}\in\mathcal{C}^3(\mathbb{R}^2;\mathbb{E}^3)$ with $\tilde{\bm{\theta}}=\bm{\theta}$ on $\overline{\omega}$. As the immersion condition and the strict inequality $\tilde{\bm{\theta}}\cdot\bm{q}>0$ hold on the compact set $\overline{\omega}$, there is a bounded open set $\tilde{\omega}$ such that $\overline{\omega}\subset\subset\tilde{\omega}$, such that
		\begin{equation*}
		\min_{y\in\overline{\tilde{\omega}}}(\tilde{\bm{\theta}}(y)\cdot\bm{q})>0,
		\end{equation*}
		and such that $\tilde{\bm{\theta}}|_{\overline{\tilde{\omega}}}$ is still an immersion (cf., e.g., Lemma~6.3 in~\cite{MeiPie2024}). We denote by $\tilde{\bm{a}}^i$ the vectors of the contravariant basis for the immersion $\tilde{\bm{\theta}}$.
		
		Denoting by $\rho$ the \emph{original} mollifier defined on page~108 of~\cite{Brez11}, for each integer $k\ge 1$ such that $k^{-1}<\frac{1}{2}\min\left\{\textup{dist}(\partial\omega;\textup{supp }\tilde{\bm{\eta}}),\textup{dist}(\partial\omega;\partial\tilde{\omega})\right\}$, we define $\rho_{1/k}$ as follows:
		\begin{equation*}
			\rho_{1/k}(z):=\dfrac{k^2}{\|\rho\|_{L^1(B(0;1))}}\rho(kz),\quad\textup{ for all }z\in\mathbb{R}^2.
		\end{equation*}
		
		It is well-known (cf., e.g., \cite{Brez11}) that $\rho_{1/k}\ge 0$, that $\textup{supp } \rho_{1/k}$ is contained in $B(0;1/k)$ and that $\int_{\mathbb{R}^2}\rho_{1/k}(z)\mathrm{d} z=1$.
		Define the functions:
		\begin{equation*}
			\tilde{\eta}_{i,k}:=\rho_{1/k}\star\tilde{\eta}_i \in \mathcal{D}(\mathbb{R}^2).
		\end{equation*}
		
		It results that $\tilde{\eta}_{\beta,k}\to\tilde{\eta}_\beta$ in $H^1(\mathbb{R}^2)$ as $k\to\infty$, and that $\tilde{\eta}_{3,k}\to\tilde{\eta}_3$ in $L^2(\mathbb{R}^2)$ as $k\to\infty$. For $y\in\omega$, set $\eta_{i,k}(y):=\tilde{\eta}_{i,k}(y)$. For a.a. $y\in\omega$, it results
		\begin{equation}
			\label{est1}
			\begin{aligned}
				&\left(\bm{\theta}(y)+\eta_{i,k}(y)\bm{a}^i(y)\right)\cdot\bm{q}
				=\left(\bm{\theta}(y)+\left(\int_{\mathbb{R}^2}\rho_{1/k}(z)\tilde{\eta}_i(y-z)\mathrm{d} z\right)\bm{a}^i(y)\right)\cdot\bm{q}\\
				&=\int_{\mathbb{R}^2}\left(\left(\tilde{\bm{\theta}}(y)-\tilde{\bm{\theta}}(y-z)\right)\cdot\bm{q}\right)\rho_{1/k}(z)\mathrm{d} z+\int_{\mathbb{R}^2}\left((\tilde{\bm{\theta}}(y-z)+\tilde{\eta}_i(y-z)\tilde{\bm{a}}^i(y-z))\cdot\bm{q}\right)\rho_{1/k}(z)\mathrm{d} z\\
				&\quad+\int_{\mathbb{R}^2}\left(\left(\tilde{\bm{a}}^i(y)-\tilde{\bm{a}}^i(y-z)\right)\cdot\bm{q}\right)\tilde{\eta}_i(y-z)\rho_{1/k}(z)\mathrm{d} z\\
				&\ge-\max_{z\in B(0;1/k)}|\tilde{\bm{\theta}}(y)-\tilde{\bm{\theta}}(y-z)|+\int_{\mathbb{R}^2}\left((\tilde{\bm{\theta}}(y-z)+\tilde{\eta}_i(y-z)\tilde{\bm{a}}^i(y-z))\cdot\bm{q}\right)\rho_{1/k}(z)\mathrm{d} z\\
				&\quad+\int_{\mathbb{R}^2}\left(\left(\tilde{\bm{a}}^i(y)-\tilde{\bm{a}}^i(y-z)\right)\cdot\bm{q}\right)\tilde{\eta}_i(y-z)\rho_{1/k}(z)\mathrm{d} z\\
				&\ge-\max_{z\in B(0;1/k)}|\tilde{\bm{\theta}}(y)-\tilde{\bm{\theta}}(y-z)|+\tilde{\delta}+\int_{\mathbb{R}^2}\left(\left(\tilde{\bm{a}}^i(y)-\tilde{\bm{a}}^i(y-z)\right)\cdot\bm{q}\right)\tilde{\eta}_i(y-z)\rho_{1/k}(z)\mathrm{d} z,
			\end{aligned}
		\end{equation}
		where $\tilde{\delta}:=\min\{\delta,\min_{y\in\overline{\tilde{\omega}}}(\tilde{\bm{\theta}}(y)\cdot\bm{q})\}$.
		
		Let us estimate the last integral term in~\eqref{est1}. The fact that $\tilde{\bm{a}}^i\in\mathcal{C}^2(\overline{\tilde{\omega}};\mathbb{E}^3)$ gives that $\tilde{\bm{a}}^i$ is Lipschitz continuous. Therefore, we obtain that there exists a constant $L>0$ independent of $i$ and $k$ such that:
		\begin{equation}
			\label{est2}
			|\tilde{\bm{a}}^i(y-z)-\tilde{\bm{a}}^i(y)|\le \dfrac{L}{k},\quad\textup{ for all }y\in\overline{\omega}, \textup{ for all }z\in B(0;1/k) \textup{ and for all }i\in\{1,2,3\}.
		\end{equation}
		
		An application of~\eqref{est2} and the Cauchy--Schwarz inequality gives:
		\begin{equation}
			\label{est3}
			\begin{aligned}
				&\left|\int_{\mathbb{R}^2}\left(\left(\tilde{\bm{a}}^i(y)-\tilde{\bm{a}}^i(y-z)\right)\cdot\bm{q}\right)\tilde{\eta}_i(y-z)\rho_{1/k}(z)\mathrm{d} z\right|\\
				&\le\dfrac{L}{k}\sum_{i=1}^{3}\left(\int_{B(y;1/k)}|\tilde{\eta}_i(w)|^2\mathrm{d}w\right)^{1/2}\left(\int_{B(0;1/k)}|\rho_{1/k}(z)|^2\mathrm{d} z\right)^{1/2}.
			\end{aligned}
		\end{equation}
		
		In light of the change of variables
		\begin{equation*}
			\int_{B(0;1/k)}|\rho_{1/k}(z)|^2\mathrm{d} z=\dfrac{k^4}{\|\rho\|_{L^1(B(0;1))}^2}\int_{B(0;1)}|\rho(z')|^2 k^{-2}\mathrm{d} z'=\dfrac{k^2\|\rho\|_{L^2(B(0;1))}^2}{\|\rho\|_{L^1(B(0;1))}^2},
		\end{equation*}
		we can thus estimate each addend on the right-hand side of~\eqref{est3} as follows:
		\begin{equation}
			\label{est4}
			\dfrac{L}{k}\left(\int_{B(y;1/k)}|\tilde{\eta}_i(z)|^2\mathrm{d} z\right)^{1/2}\left(\int_{B(0;1/k)}|\rho_{1/k}(z)|^2\mathrm{d} z\right)^{1/2}
			\le L\left(\int_{B(y;1/k)}|\tilde{\eta}_i(z)|^2\mathrm{d} z\right)^{1/2}\dfrac{\|\rho\|_{L^2(B(0;1))}}{\|\rho\|_{L^1(B(0;1))}}.
		\end{equation}
		
		Since $\max_{y\in\overline{\omega}}|B(y;1/k)|=\pi k^{-2}\to 0$ as $k\to\infty$, the absolute continuity of the Lebesgue integral gives that:
		\begin{equation}
			\label{est4bis}
			\int_{B(y;1/k)}|\tilde{\eta}_i(w)|^2\mathrm{d}w \to 0,\quad\textup{ as }k\to\infty \textup{ independently of }y\in\omega.
		\end{equation}
		
		Thanks to~\eqref{est3} and~\eqref{est4}, the right-hand side of~\eqref{est1} can be estimated as follows:
		\begin{equation}
			\label{est5}
			\begin{aligned}
				&-\max_{z\in B(0;1/k)}|\tilde{\bm{\theta}}(y)-\tilde{\bm{\theta}}(y-z)|+\tilde{\delta}+\int_{\mathbb{R}^2}\left(\left(\tilde{\bm{a}}^i(y)-\tilde{\bm{a}}^i(y-z)\right)\cdot\bm{q}\right)\tilde{\eta}_i(y-z)\rho_{1/k}(z)\mathrm{d} z\\
				&\ge-\left(\max_{z\in B(0;1/k)}|\tilde{\bm{\theta}}(y)-\tilde{\bm{\theta}}(y-z)|\right)-L\sum_{i=1}^3\left(\int_{B(y;1/k)}|\tilde{\eta}_i(z)|^2\mathrm{d} z\right)^{1/2}\dfrac{\|\rho\|_{L^2(B(0;1))}}{\|\rho\|_{L^1(B(0;1))}}+\tilde{\delta}.
			\end{aligned}
		\end{equation}
		
		For $k$ sufficiently large, the uniform continuity of $\tilde{\bm{\theta}}$ on $\overline{\tilde{\omega}}$ and~\eqref{est4bis} imply that the two subtracted terms on the right-hand side of~\eqref{est5} tend to zero uniformly in $y\in\omega$. Hence, for a.a. $y\in\omega$,
		\begin{equation*}
			\left(\bm{\theta}(y)+\eta_{i,k}(y)\bm{a}^i(y)\right)\cdot\bm{q}\ge\dfrac{\tilde{\delta}}{2}>0.
		\end{equation*}
		
		The choice of $k$ we made beforehand ensures that $\textup{supp }\tilde{\eta}_{i,k}$ is compact in $\omega$, so that $\eta_{i,k}:=\tilde{\eta}_{i,k}|_{\omega}\in\mathcal{D}(\omega)$. Thus $\bm{\eta}_k:=(\eta_{i,k})\in\bm{\mathcal{D}}(\omega)\cap\bm{U}_M^{(2)}(\omega)=\bm{U}_M^{(3)}(\omega)$. Since $\tilde{\eta}_{\beta,k}\to\tilde{\eta}_\beta$ in $H^1(\mathbb{R}^2)$ as $k\to\infty$ and $\tilde{\eta}_{3,k}\to\tilde{\eta}_3$ in $L^2(\mathbb{R}^2)$ as $k\to\infty$, we obtain that $\bm{\eta}_k\to\bm{\eta}$ in $H^1(\omega)\times H^1(\omega)\times L^2(\omega)$ as $k\to\infty$, which proves that $\bm{U}_M^{(3)}(\omega)$ is dense in $\bm{U}_M^{(2)}(\omega)$.
		
		Combining $(i)$--$(vi)$, and noting that $\bm{U}_M^{(3)}(\omega)\subseteq\bm{U}_M(\omega)\cap\bm{\mathcal{D}}(\omega)\subseteq\bm{U}_M(\omega)$, the theorem follows.
\end{proof}

\bibliographystyle{abbrvnat} 
\bibliography{references.bib}	

@book{Ciarlet2000,
      author = "Ciarlet, P G",
       title = "Mathematical Elasticity. Vol. III: {T}heory of Shells",
   publisher = "North-Holland",
     address = "Amsterdam",
        year = "2000"
}

@article{CiaMarPie2018,
	author="Ciarlet, P. G and Mardare, C and Piersanti, P",
	title="An obstacle problem for elliptic membrane shells",
	journal="{M}ath. {M}ech. {S}olids",
	volume="24",
	number="5",
	year="2019",
	pages="1503--1529"
}

@book{Brez11,
	author="Brezis, H",
	title="Functional {A}nalysis, {S}obolev {S}paces and {P}artial {D}ifferential {E}quations",
	publisher="Springer",
	address="New York",
	year="2011"
}

@article{CiaMarPie2018b,
	author="Ciarlet, P G and Mardare, C and Piersanti, P",
	title="Un probl\`eme de confinement pour une coque membranaire lin\'eairement \'elastique de type elliptique",
	journal="C.R. Acad. Sci. Paris, S\'{e}r. I",
	volume="356",
	number="10",
	pages="1040--1051",
	year="2018"
}

@article{Pie-2021,
	AUTHOR = {Piersanti, Paolo},
	TITLE = {On the justification of the frictionless time-dependent {K}oiter's model for thermoelastic shells},
	JOURNAL = {J. Differential Equations},
	VOLUME = {296},
	YEAR = {2021},
	PAGES = {50--106}
}

@book{Ciarlet2025,
	author={Ciarlet, P. G.},
	title={Linear and Nonlinear Functional Analysis with Applications},
	publisher={Society for Industrial and Applied Mathematics},
	edition={Second},
	address={Philadelphia},
	year={2025}
}

@book{Leoni2017,
	AUTHOR = {Leoni, Giovanni},
	TITLE = {A first course in {S}obolev spaces},
	SERIES = {Graduate Studies in Mathematics},
	VOLUME = {181},
	EDITION = {Second},
	PUBLISHER = {American Mathematical Society, Providence, RI},
	YEAR = {2017},
	PAGES = {xxii+734}
}

@article{MeiPie2024,
	AUTHOR = {Meixner, Aaron and Piersanti, Paolo},
	TITLE = {Numerical approximation of the solution of an obstacle problem modelling the displacement of elliptic membrane shells via the penalty method},
	JOURNAL = {Appl. Math. Optim.},
	VOLUME = {89},
	YEAR = {2024},
	NUMBER = {2},
	PAGES = {Paper No. 45, 60},
}

@article{Pie-2021-2,
	AUTHOR = {Piersanti, Paolo},
	TITLE = {On the improved interior regularity of the solution of a second order elliptic boundary value problem modelling the displacement of a linearly elastic elliptic membrane shell subject to an obstacle},
	JOURNAL = {Discrete Contin. Dyn. Syst.},
	volume={42},
	number={2},
	YEAR = {2022},
	pages={1011--1037}
}

\end{document}